\documentclass[preprint,11pt,sort&compress,numbers]{elsarticle}

\usepackage{amsmath,amsthm,amsfonts,amssymb,latexsym,mathrsfs,color,cases,url,enumerate}

\newtheorem{thm}{Theorem}[section]
\newtheorem{lem}[thm]{Lemma}
\newtheorem{coro}[thm]{Corollary}
\newtheorem{prop}[thm]{Proposition}

\newtheorem*{PT}{Poincar\'e Theorem}
\newtheorem*{DJI}{Desnanot-Jacobi Determinant Identity}
\newtheorem*{PiT}{Pincherle Theorem}

\theoremstyle{definition}

\newtheorem{exm}[thm]{Example}

\newtheorem{rem}[thm]{Remark}

\newcommand{\lrf}[1]{\lfloor #1\rfloor}
\def\la{\lambda}

\def\b{\beta}%{\widetilde{b}}
\def\c{\gamma}%{\widetilde{b}}

\numberwithin{equation}{section}
\journal{arXiv}%{LAA}

\usepackage{geometry}
\begin{document}

\begin{frontmatter}

\title{Positivity problem of three-term recurrence sequences}
%{Positivity of Taylor coefficients of rational functions}
\author{Yanni Pei\corref{cor1}}
\ead{peiyanni@hotmail.com}
\author{Yaling Wang\corref{cor2}}
\ead{wang-yaling@hotmail.com}
\author{Yi Wang\corref{cor3}}
\ead{wangyi@dlut.edu.cn}
\cortext[cor3]{Corresponding author.}
\address{School of Mathematical Sciences, Dalian University of Technology, Dalian 116024, P.R. China}
\date{}
\begin{abstract}
We present some necessary and/or sufficient conditions for the positivity problem of three-term recurrence sequences.
As applications we show the positivity of diagonal Taylor coefficients of some rational functions
in a unified approach.
We also establish a criterion for the positivity and log-convexity of such sequences.
%diagonal coefficients of some rational functions in a unified approach.
\end{abstract}

\end{frontmatter}

\section{Introduction}

Let $(u_n)_{n\ge 0}$ be a sequence of real numbers satisfying the three-term recurrence relation
\begin{equation}\label{u-rr}
a(n)u_{n+1}=b(n)u_{n}-c(n)u_{n-1},\qquad n=1,2,\ldots,
\end{equation}
where $a(n),b(n),c(n)$ %are polynomials in $n$ and
take positive values for all $n\ge 1$.
We say also that $u_n$ is a {\it solution} of the difference equation \eqref{u-rr}.
  %the second-order linear homogeneous
  %the three-term difference equation \eqref{u-rr}.
The positivity problem naturally arises:
in which case, the three-term recurrence sequence is positive?
Such a problem is closely related to the total nonnegativity of matrices.
Following \cite{FJ11},
we say that a (finite or infinite) matrix is {\it totally nonnegative} (TN for short),
if its minors of all orders are nonnegative.
We have the following characterization.

\begin{thm}[Characterization]
\label{pp-TN}
Let $u_n$ be a solution of the difference equation \eqref{u-rr}. %$a(n)u_{n+1}=b(n)u_n-c(n)u_{n-1}$.
Then $(u_n)_{n\ge 0}$ is positive if and only if $u_0>0$ and the tridiagonal matrix
$$
M_0=\left(
    \begin{array}{ccccc}
      u_1 & c(1) &  &  &  \\
      u_0 & b(1) & c(2) &  &  \\
       & a(1) & b(2) & c(3) &  \\
       &  & a(2) & b(3) & \ddots \\
       &  &  & \ddots & \ddots \\
    \end{array}
  \right).
$$
is totally nonnegative.
\end{thm}

Our interest in the positivity problem of three-term recurrence sequences is motivated by
the positivity of diagonal Taylor coefficients of multivariate rational functions
(see Example \ref{dtc-exm}).
%(see \cite{Kau07,Pil19,SS14,Str08,SZ15} and Example \ref{dtc-exm}).
%we will present a simple and effective criterion for the positivity problem.
Some of diagonal coefficients are the so-called Ap\'ery-like numbers that satisfy three-term recurrence relations,
in which $a(n),b(n),c(n)$ are all quadratic polynomials in $n$
or are all cubic polynomials in $n$.
%(see Remark \ref{aln}).
%$(n+1)^2u_{n+1}=(rn^2+rn+s)u_n-tn^2u_{n-1}$ or
%$(n+1)^3u_{n+1}=(2n+1)(rn^2+rn+s)u_n-tn^3u_{n-1}$,
%with $u_0=1$ and $u_1=s$.

%Since our interest,
Throughout this paper,
we always assume that %the coefficients
$a(n),b(n),c(n)$ in \eqref{u-rr} are polynomials in $n$ with the same degree $\delta$ and
$$a(n)=an^\delta+a'n^{\delta-1}+\cdots,\quad
b(n)=bn^\delta+b'n^{\delta-1}+\cdots,\quad
c(n)=cn^\delta+c'n^{\delta-1}+\cdots,$$
where the leading coefficients $a,b,c$ are positive.

Following Elaydi \cite{Ela05},
%a nontrivial solution $u_n$ of %the three-term difference equation
a nontrivial solution $u_n$ of
\eqref{u-rr} is said to be {\it oscillatory} (around zero)
if for every positive integer $N$ there exists $n\ge N$ such that $u_nu_{n+1}\le 0$.
%Otherwise, the sequence is said to be {\it nonoscillatory}.
Otherwise, the solution is said to be {\it nonoscillatory}.
In other words,
a solution is nonoscillatory if it is {\it eventually sign-definite},
i.e., either eventually positive or eventually negative.
 %In other words,
 %$(u_n)$ is oscillatory if it is neither eventually positive nor eventually negative.
We say that a nontrivial solution $u^*_n$ is a {\it minimal solution} of \eqref{u-rr}
if $\lim_{n\rightarrow +\infty}u^*_n/u_n=0$ for arbitrary solution $u_n$ of \eqref{u-rr}
that is not a multiple of $u^*_n$.
Clearly, a minimal solution is unique up to multiplicity.
%As pointed out in \cite{Ela05},
Minimal solutions play a central role in the convergence of continued fractions
and the asymptotics of orthogonal polynomials.
%Throughout this paper
For convenience, we also write \eqref{u-rr} as %in the normal form
\begin{equation}\label{u-rr-1}
u_{n+1}=\b_nu_{n}-\c_nu_{n-1},\qquad n=1,2,\ldots
\end{equation}
where $\b_n=b(n)/a(n)$ and $\c_n=c(n)/a(n)$.
%Clearly, \eqref{u-rr-1} shares the same characteristic roots with \eqref{u-rr}.
For simplicity,
we denote the continued fraction in a compact form
\begin{equation}\label{cf}
\frac{\c_1}{\b_1-}\;\frac{\c_2}{\b_2-}\;\frac{\c_3}{\b_3-}\;\cdots
:=\cfrac{\c_1}{\b_1-\cfrac{\c_2}{\b_2-\cfrac{\c_3}{\b_3-\cdots}}}
\end{equation}

\begin{thm}[Necessity]
\label{nc-pp}
Let $(u_n)_{n\ge 0}$ be a solution of \eqref{u-rr}.
\begin{enumerate}[\rm (i)]
\item
If $(u_n)$ is eventually sign-definite, %eventually positive or eventually negative
then $b^2\ge 4ac$.
\item
If $(u_n)_{n\ge 0}$ is positive,
then the continued fraction \eqref{cf}
converges to a finite positive limit $\rho_0$ and
$u_1\ge\rho_0 u_0$.
Moreover,
the solution $(u^*_n)_{n\ge 0}$ of \eqref{u-rr} decided by $u^*_0=1$ and $u^*_1=\rho_0$
is a positive and minimal solution of \eqref{u-rr}.
\end{enumerate}
\end{thm}

\begin{thm}[Sufficiency]\label{sc-epp}
If $b^2>4ac$,
then each nontrivial solution $(u_n)$ of \eqref{u-rr} is eventually sign-definite.
\end{thm}

%In the case $b^2=4ac$,
%a solution of \eqref{u-rr} may be either eventually sign-definite or oscillatory.
%%See Example \ref{D=0}.
%For example, every solution of the difference equation $u_{n+1}=2u_n-u_{n-1}$ is eventually sign-definite
%since $u_n=n(u_1-u_0)+u_0$.
%On the other hand, the difference equation
%$nu_{n+1}=(2n-1)u_n-nu_{n-1}$ is oscillatory \cite[Exercise 7.1.3]{Ela05}.

Denote the characteristic polynomial of the difference equation \eqref{u-rr} by $Q(\la)=a\la^2-b\la+c$
%the discriminant by $\Delta=b^2-4ac$
and the characteristic roots by
\begin{equation*}\label{roots}
\la_1=\frac{b-\sqrt{b^2-4ac}}{2a},\quad
\la_2=\frac{b+\sqrt{b^2-4ac}}{2a}.
\end{equation*}
Denote
$Q_n(\la):=a(n)\la^2-b(n)\la+c(n)$. %, \Delta_n, \la_1^{(n)}$ and $\la_1^{(n)}$.
Then $Q_n(\la)=Q(\la)n^\delta+\cdots$.
Assume that $b^2>4ac$.
Then for $\la_1<\la_0<\la_2$, we have $Q(\la_0)<0$,
and so $Q_n(\la_0)<0$ for sufficiently large $n$.
%Conversely,
%if $Q_n(\la_0)\le 0$ for sufficiently large $n$,
%then $Q(\la_0)\le 0$,
%and so that %$\Delta\ge 0$ and
%$b^2\ge 4ac$ and $\la_1\le\la_0\le\la_2$.

\begin{thm}[Criterion]
\label{sc-pp}
Let $(u_n)_{n\ge 0}$ be a solution of \eqref{u-rr}.
Assume that these exists a positive number $\la_0$ such that
$Q_n(\la_0)\le 0$ for all $n\ge m$ and
$u_{m+1}\ge \la_0u_m>0$.
Then $(u_n)_{n\ge m}$ is positive.
\end{thm}

Clearly, $Q_n(\la_0)\le 0$ for all $n\ge m$
implies that $Q(\la_0)\le 0$,
and so that %$\Delta\ge 0$ and
$b^2\ge 4ac$ and $\la_1\le\la_0\le\la_2$.
In this sense, %sufficient conditions for the positivity in Theorem \ref{sc-pp} are ``almost" necessary.
the conditions in Theorem \ref{sc-pp} are ``almost" necessary.

This paper is organized as follows.
In \S 2,
we show Theorems \ref{pp-TN} and {\ref{nc-pp} by means of the total nonnegativity of tridiagonal matrices
and the theory of continued fractions.
In \S 3,
we first present the proofs of Theorems \ref{sc-epp} and \ref{sc-pp},
and then apply them to show the positivity of diagonal Taylor coefficients of some famous rational functions.
We also establish a criterion for the positivity and log-convexity of three-term recurrence sequences.
%As applications we show the log-convexity of some Ap\'ery-like numbers in a unified approach.
In \S 4,
we illustrate that the difference equation \eqref{u-rr} may be either oscillatory or nonoscillatory in the case $b^2=4ac$.
We also propose a couple problems for further work.

\section{Proof of Theorems \ref{pp-TN} and \ref{nc-pp}}%Necessary conditions}

Asymptotic behavior of solutions of second-order difference equations
has been extensively and deeply investigated
(see \cite[Chapter 8]{Ela05} for instance).
%\cite{BT33,Ela05,WZ85,WL92a,WL92b}.
%every homogeneous second-order difference equation gives rise to
%an associated tridiagonal matrix.
However, as will be seen below,
the total nonnegativity of (tridiagonal) matrices is a more natural approach to the positivity problem.
 %and continued fraction.
 %(see \cite[p. 399]{Ela05} for instance).
%We will prove Theorem \ref{nc-pp} (i)
%from the viewpoint of the total nonnegativity of matrices.

%\subsection{\bf Proof of Theorem \ref{nc-pp} (i)}

Following \cite{FJ11},
we say that a (finite or infinite) matrix is {\it totally nonnegative} (TN for short),
if its minors of all orders are nonnegative.
Let $(a_n)_{n\ge 0}$ be an infinite sequence of nonnegative numbers.
It is called a {\it P\'olya frequency} (PF for short) sequence
if the associated Toeplitz matrix
$$[a_{i-j}]_{i,j\ge 0}=\left[
\begin{array}{lllll}
a_{0} &  &  &  &\\
a_{1} & a_{0} & &\\
a_{2} & a_{1} & a_{0} &  &\\
a_{3} & a_{2} & a_{1} & a_{0} &\\
\vdots & & & & \ddots\\
\end{array}
\right]$$
is TN.
We say that a finite sequence $(a_0,a_1,\ldots, a_n)$ is PF
if the corresponding infinite sequence $(a_0,a_1,\ldots,a_n,0,0,\ldots)$ is PF.
%The fundamental representation theorem of Schoenberg and Edrei states
%that a sequence $(a_n)_{n\ge 0}$ of real numbers is PF if and only if its generating function has the form
%$$\sum_{n\ge 0}a_nz^n=\frac{\prod_{j\ge 1}(1+\alpha_jz)}{\prod_{j\ge 1}(1-\beta_jz)}e^{\gamma z}$$
%in some open disk centered at the origin,
%where $\alpha_j,\beta_j,\gamma\ge 0$ and $\sum_{j\ge 1}(\alpha_j + \beta_j) < +\infty$
%(see \cite[p. 412]{Kar68} for instance).
%In particular,
A classical result of Aissen, Schoenberg and Whitney states that
a finite sequence of nonnegative numbers is PF
if and only if its generating function has only real zeros
(see \cite[p. 399]{Kar68} for instance).
For example, the sequence $(r,s,t)$ of nonnegative numbers is PF if and only if $s^2\ge 4rt$.

To prove Theorem \ref{pp-TN},
we need the following result (see \cite[Example 2.2, p.149]{Min88} for instance).

\begin{lem}%[{\cite[Example 2.2, p.149]{Min88}}]
\label{irr-tp}
An irreducible nonnegative tridiagonal matrix is totally nonnegative
if and only if all its leading principal minors are positive.
\end{lem}

\begin{proof}[Proof of Theorem \ref{pp-TN}]
Let $(u_n)_{n\ge 0}$ be a solution of the difference equation \eqref{u-rr-1}.
Then $u_{n+1}=\b_nu_n-\c_nu_{n-1}$,
where $\b_n=b(n)/a(n)$ and $\c_n=c(n)/a(n)$.
Denote the infinite tridiagonal matrix
$$
M_1=\left(
    \begin{array}{ccccc}
      u_1 & \c_1 &  &  &  \\
      u_0 & \b_1 & \c_2 &  &  \\
       & 1 & \b_2 & \c_3 &  \\
       &  & 1 & \ddots & \ddots \\
       &  &  & \ddots & \ddots \\
    \end{array}
  \right).
$$
Then for $n\ge 1$,
the $n$th leading principal minor of $M_1$ is precisely $u_{n}$,
since they satisfy the same three-term recurrence relation.
So, if $u_0$ and $\gamma_n$ are positive for all $n\ge 1$,
then the sequence $(u_n)_{n\ge 1}$ is positive if and only if the tridiagonal matrix $M_1$ is totally nonnegative
by Lemma \ref{irr-tp}.
Clearly, $M_0$ is TN if and only if $M_1$ is TN.
%Then $u_{n}$ is precisely the $n$th leading principal minor of $M_0$ for $n\ge 1$.
%It is known that an irreducible nonnegative tridiagonal matrix is totally nonnegative
%if and only if all its leading principal minors are positive
%(see Minc's book \cite[Example 2.2, p.149]{Min88} for instance).
Hence the positivity of the sequence $(u_n)_{n\ge 1}$
is equivalent to the total nonnegativity of the tridiagonal matrix $M_0$.
This completes the proof of Theorem \ref{pp-TN}.
%characterization for the positivity problem of three-term recurrence sequences.
\end{proof}

There are characterizations for the total nonnegativity of tridiagonal matrices besides Lemma \ref{irr-tp}.

\begin{lem}[{\cite[Example 2.1, p.147]{Min88}}]\label{tri-tp}
A nonnegative tridiagonal matrix is totally nonnegative if and only if all its principal minors are nonnegative.
\end{lem}

\begin{lem}[{\cite[Theorem 4.3]{Pin10}}]\label{tri-tn}
A nonnegative tridiagonal matrix is totally nonnegative if and only if
all its principal minors containing consecutive rows and columns are nonnegative,
\end{lem}

We also refer the reader to \cite{CLW15ra,CLW15rm,LMW16,WZ16} for some criteria for the total nonnegativity of tridiagonal matrices.

%We are now in a position to prove Theorem \ref{nc-pp} (i).

\begin{proof}[\bf Proof of Theorem \ref{nc-pp} (i)]
Clearly, it suffices to consider the case the total sequence $(u_n)_{n\ge 0}$ is positive.
By Proposition \ref{pp-TN},
to prove Theorem \ref{nc-pp} (i),
it suffices to prove that the total nonnegativity of the matrix $M_0$ implies $b^2\ge 4ac$.
In other words,
we need to prove that the sequence $(c,b,a)$ is a P\'olya frequency sequence,
or equivalently, the tridiagonal matrix
%It suffices to show that the Toeplitz matrix
%the tridiagonal matrix
$$
\left(
    \begin{array}{ccccc}
      b & c &  &  &  \\
      a & b & c &  &  \\
       & a & b & c &  \\
       &  & a & b & \ddots \\
       &  &  & \ddots & \ddots \\
    \end{array}
  \right)
$$
is totally nonnegative.
By Lemma \ref{tri-tn},
it suffices to show that the determinants
$$D_k=
\det\left(
    \begin{array}{ccccc}
      b & c &  &  &  \\
      a & b & c &  &  \\
       & a & b & \ddots &  \\
       &  &\ddots &\ddots & c\\
       &  &  & a& b\\
    \end{array}
  \right)_{k\times k}
$$
are nonnegative for all $k\ge 1$.

Suppose the contrary and assume that $D_m<0$ for some $m\ge 1$.
Consider the determinants
$$
D_m(n)=\det\left(
    \begin{array}{ccccc}
      b(n+1) & c(n+2) &  &  &  \\
      a(n+1) & b(n+2) & c(n+3) &  &  \\
       & a(n+2) & b(n+3) & \ddots &  \\
       &  &\ddots &\ddots & c(n+m)\\
       &  &  & a(n+m-1) & b(n+m)\\
    \end{array}
  \right)_{m\times m}.
$$
Clearly, $D_m(n)\ge 0$ for all $n\ge 0$
since they are minors of the totally nonnegative matrix $M_0$.
%by the total nonnegativity of the matrix $M_0$.
On the other hand,
%Note that the determinant of the left hand side
note that
$D_m(n)$
are polynomials in $n$
of degree $m\delta$ with the leading coefficient $D_m$:
$$D_m(n)=D_mn^{m\delta}+\cdots.$$
It follows that $D_m(n)<0$ for sufficiently large $n$, a contradiction.

Thus $D_k\ge 0$ for all $k\ge 1$, as desired.
This completes the proof of Theorem \ref{nc-pp} (i).
%and $J$ is therefore totally nonnegative.
%Thus $b^2\ge 4ac$, as desired.
\end{proof}

%\subsection{\bf Proof of Theorem \ref{nc-pp} (ii)}

To prove Theorem \ref{nc-pp} (ii),
we need the following classical determinant evaluation rule.
%(see Zeilberger \cite{Zil97} for a combinatorial proof).

\begin{DJI}\label{dji}
Let the matrix $M=[m_{ij}]_{0\le i,j\le k}$. %be an $m\times m$ matrix.
Then $$\det M\cdot\det M^{0,k}_{0,k}=\det M_k^k\cdot\det M_0^0-\det M_0^k\cdot\det M_k^0,$$
where $M^I_J$ denote the submatrix obtained from $M$ by deleting those rows in $I$ and columns in $J$.
\end{DJI}

Let $\b=(\b_n)_{n\ge 0}$ and $\c=(\c_n)_{n\ge 1}$ be two sequences of positive numbers.
Denote
$$
J_i=\left(
    \begin{array}{ccccc}
      \b_i & \c_{i+1} &  &  &  \\
      1 & \b_{i+1} & \c_{i+2} &  &  \\
       & 1 & \b_{i+2} & \c_{i+3} &  \\
       &  & 1 & \b_{i+3} & \ddots \\
       &  &  & \ddots & \ddots \\
    \end{array}
  \right),\quad i=0,1,2,\ldots.
$$

\begin{lem}\label{td-cf}
If the tridiagonal matrix $J_0$ is totally nonnegative,
then the continued fraction
\begin{equation}\label{cf-0}
\b_0-%\cfrac{\c_1}{\b_1-\cfrac{\c_2}{\b_2-\cfrac{\c_3}{\b_3-\cdots}}}
\frac{\c_1}{\b_1-}\;\frac{\c_2}{\b_2-}\;\frac{\c_3}{\b_3-}\;\frac{\c_4}{\b_4-}\;\cdots
\end{equation}
is convergent. %to a nonnegative number.
\end{lem}
\begin{proof}
Let $A(n)$ and $B(n)$ be the $n$th partial numerator and
the $n$th partial denominator of the continued fraction \eqref{cf-0}.
Then we have
\begin{eqnarray*}
  A(n)=\b_nA(n-1)-\c_nA(n-2),&& A(-1)=1,\ A(0)=\b_0; \\
  B(n)=\b_nB(n-1)-\c_nB(n-2),&& B(-1)=0,\ B(0)=1
\end{eqnarray*}
by the fundamental recurrence formula for continued fractions
(see \cite[Theorem 9.2]{Ela05} for instance).
To show that the continued fraction \eqref{cf-0} is convergent,
it suffices to show that $A(n)/B(n)$ is convergent.

For $n\ge i\ge 0$, denote
$$u_{i,n}=
\det\left(
    \begin{array}{ccccc}
      \b_i & \c_{i+1} &  &  &  \\
      1 & \b_{i+1} & \c_{i+2} &  &  \\
       & 1 & \b_{i+2} & \ddots &  \\
       &  &\ddots &\ddots & \c_{n}\\
       &  &  & 1 & \b_{n}\\
    \end{array}
  \right).
$$
If $J_0$ is TN, then so is $J_i$ for each $i\ge 0$.
Thus $u_{i,n}>0$ by Lemma \ref{irr-tp}. %for $0\le i\le n$.

%the method of Dodgeson condensation
Applying the Desnanot-Jacobi determinant identity to the determinant $u_{i,n+1}$,
we obtain
$$u_{i,n+1}u_{i+1,n}=u_{i+1,n+1}u_{i,n}-\c_{i+1}\cdots \c_n.$$
It follows that $u_{i,n+1}u_{i+1,n}<u_{i+1,n+1}u_{i,n}$. %since all $\c$ are positive.
Thus $u_{i,n}/u_{i+1,n}$ is decreasing in $n$ and is therefore convergent.
Let $\lim_{n\rightarrow +\infty} u_{i,n}/u_{i+1,n}=\ell_i$.
Clearly, $\ell_i\ge 0$.
Note that $A(n)=u_{0,n}$ and $B(n)=u_{1,n}$.
Hence $A(n)/B(n)$ is convergent,
and $\lim_{n\rightarrow +\infty}A(n)/B(n)=\ell_0$.
%as desired.
\end{proof}

\begin{rem}\label{rem-cf}
We have showed that the continued fraction \eqref{cf-0} converges to $\ell_0$.
%As an immediate consequence,
%Similarly,
More generally, we have
\begin{equation}\label{cf-i}
\ell_i=\b_i-%\cfrac{\c_{i+1}}{\b_{i+1}-\cfrac{\c_{i+2}}{\b_{i+2}-\cdots}}
\frac{\c_{i+1}}{\b_{i+1}-}\;\frac{\c_{i+2}}{\b_{i+2}-}\;\frac{\c_{i+3}}{\b_{i+3}-}\;\frac{\c_{i+4}}{\b_{i+4}-}\;\cdots
\end{equation}
for $i\ge 0$.
Clearly, $\ell_i=\b_i-\frac{\c_{i+1}}{\ell_{i+1}}$.
Hence $\ell_{i+1}\neq 0$, and so $\ell_{i+1}>0$ for $i\ge 0$.
Denote
\begin{equation}\label{ri}
%\b_i-%\cfrac{\c_{i+1}}{\b_{i+1}-\cfrac{\c_{i+2}}{\b_{i+2}-\cdots}}
\rho_i=\frac{\c_{i+1}}{\b_{i+1}-}\;\frac{\c_{i+2}}{\b_{i+2}-}\;\frac{\c_{i+3}}{\b_{i+3}-}\;\frac{\c_{i+4}}{\b_{i+4}-}\;\cdots.
\end{equation}
Then $\rho_i=\frac{\c_{i+1}}{\ell_{i+1}}$.
Thus $\rho_i>0$ for $i\ge 0$.
On the other hand,
$\ell_i=\b_i-\rho_i$.
Hence $\b_0\ge\rho_0$ and $\b_{i+1}>\rho_{i+1}$ for $i\ge 0$.
\end{rem}

The following classic result was given by Pincherle in his fundamental work on continued fractions
(see \cite[Theorem 9.5]{Ela05} for instance).

\begin{PiT}
The continued fraction
%\begin{equation*}\label{cf}
$$\frac{\c_1}{\b_1-}\;\frac{\c_2}{\b_2-}\;\frac{\c_3}{\b_3-}\;\frac{\c_4}{\b_4-}\;\cdots$$
%\end{equation*}
converges if and only if
the difference equation $u_{n+1}=\b_nu_n-\c_nu_{n-1}$ %\eqref{u-rr-1}
has a minimal solution $u^*_n$ with $u^*_0=1$.
In case of convergence, moreover, one has
\begin{equation*}\label{cf-n}
\frac{u^*_{n+1}}{u^*_{n}}=
\frac{\c_{n+1}}{\b_{n+1}-}\;\frac{\c_{n+2}}{\b_{n+2}-}\;\frac{\c_{n+3}}{\b_{n+3}-}\;\frac{\c_{n+4}}{\b_{n+4}-}\;\cdots.
\end{equation*}
\end{PiT}

%We are now in a position to prove Theorem \ref{nc-pp} (ii).

\begin{proof}[\bf Proof of Theorem \ref{nc-pp} (ii)]

Let $(u_n)_{n\ge 0}$ be a positive solution of the difference equation
$u_{n+1}=\b_n u_n-\c_n u_{n-1}$ and $\b_0=u_1/u_0$.
Then the tridiagonal matrix $J_0$ is totally nonnegative.
By Remark \ref{rem-cf}, we have $\b_0\ge\rho_0>0$,
and so $u_1\ge\rho_0 u_0$.

On the other hand,
we have $u^*_{n+1}=\rho_{n}u^*_{n}$ by Lemma \ref{td-cf} and Pincherle Theorem,
and $\rho_n>0$ again by Remark \ref{rem-cf}.
Thus the solution $(u^*_n)_{n\ge 0}$ of \eqref{u-rr} decided by $u^*_0=1$ and $u^*_1=\rho_0$
is a positive and minimal solution of \eqref{u-rr}.
This completes the proof of Theorem \ref{nc-pp} (ii).
\end{proof}

\section{Proofs and applications of Theorems \ref{sc-epp} and \ref{sc-pp}}%Sufficient conditions}

%Following \cite{Ela05},
We say that \eqref{u-rr} is a difference equation of {\it Poincar\'e type}
in the sense that both the sequences $b(n)/a(n)$ and $c(n)/a(n)$ have finite limit.
The following Poincar\'e theorem
marks the beginning of research in the qualitative theory of linear difference equations
(see \cite[Theorem 8.9]{Ela05} for instance).

\begin{PT}
%Let $u_n$ be a solution of the difference equation \eqref{u-rr} of Poincar\'e type.
%Suppose that the characteristic roots have distinct moduli.
Suppose that \eqref{u-rr} is a difference equation  of Poincar\'e type
and that the characteristic roots have distinct moduli.
If $u_n$ is a solution of \eqref{u-rr},
then either $u_n=0$ for all large $n$, or
$\lim_{n\rightarrow +\infty}\frac{u_{n+1}}{u_n}=\la_i$
for some characteristic root $\la_i$.
\end{PT}

\begin{proof}[\bf Proof of Theorem \ref{sc-epp}]
By Poincar\'e theorem,
$u_{n+1}/u_n\rightarrow \la_i$ for some $i$.
Now $0<\la_1<\la_2$. %by $b^2>4ac$.
Hence there exists a positive integer $N$ such that $u_{n+1}/u_n>0$ for $n\ge N$.
The sequence $(u_n)$ is therefore sign-definite.
\end{proof}

\begin{proof}[\bf Proof of Theorem \ref{sc-pp}]
Assume that $u_n\ge \la_0u_{n-1}>0$.
Then by \eqref{u-rr},
$$u_{n+1}=\frac{b(n)}{a(n)}u_n-\frac{c(n)}{a(n)}u_{n-1}
\ge\frac{b(n)}{a(n)}u_n-\frac{c(n)}{a(n)}\frac{u_n}{\la_0}
=\left[\frac{b(n)\la_0-c(n)}{a(n)\la_0}\right]u_n
\ge\la_0u_n>0.$$
Thus $(u_n)_{n\ge m}$ is positive by induction.
\end{proof}

%Two particular interest special cases of Theorem \ref{sc-pp} are the following.
A particular interest special case of Theorem \ref{sc-pp} is the following.
%The following are some special cases of Theorem \ref{sc-pp}.

\begin{coro}\label{=1}
If $b(n)\ge a(n)+c(n)$ for all $n\ge 1$ and $u_1\ge u_0>0$,
then $(u_n)_{n\ge 0}$ is positive.
\end{coro}
\begin{proof}
%The condition $b(n)\ge a(n)+c(n)$ is equivalent to $Q_n(1)\le 0$,
%and so $\la_1\le 1\le\la_2$.
The statement follows from Theorem \ref{sc-pp} by taking $\la_0=1$.
\end{proof}

%When the coefficients in \eqref{u-rr} is linear function in $n$.
%It is often convenience to take $\la_0=\la_1$,
A preferred candidate for $\la_0$ in Theorem \ref{sc-pp} is $\la_1$.
Note that
$Q_n(\la_1)$ is a polynomial in $n$ of degree less than $\delta$ and is easier to estimate.

\begin{coro}\label{1-rr}
Suppose that
$(an+a_0)u_{n+1}=(bn+b_0)u_{n}-(cn+c_0)u_{n-1}$.
Then $(u_n)_{n\ge 0}$ is positive if
$b^2\ge 4ac$,
$a_0\la_1^2-b_0\la_1+c_0\le 0$, and
$u_1\ge\la_1 u_0$.
%\quad where $\la_1=\frac{b-\sqrt{b^2-4ac}}{2a}$.
\end{coro}
\begin{proof}
We have $Q_n(\la_1)=a_0\la_1^2-b_0\la_1+c_0$,
and so the statement follows from Theorem \ref{sc-pp} by taking $\la_0=\la_1$.
\end{proof}

%The following folklore result follows from
%Theorem \ref{nc-pp} (ii) and Theorem \ref{sc-pp}.
The following folklore result is an immediate consequence of
Theorem \ref{nc-pp} and Theorem \ref{sc-pp},
which can be found in \cite{HHH06} for instance.

\begin{coro}\label{cont}
Suppose that $au_{n+1}=bu_n-cu_{n-1}$,
where $a,b,c$ are positive number.
Then $(u_n)_{n\ge 0}$ is positive
if and only if $b^2\ge 4ac$ and $u_1\ge\la_1u_0>0$.
\end{coro}
\begin{proof}
The ``if" part follows from Theorem \ref{sc-pp}.
Now assume that $(u_n)_{n\ge 0}$ is positive.
Then $b^2\ge 4ac$ and $u_1\ge\rho u_0$ from Theorem \ref{nc-pp},
where $\b=b/a,\c=c/a$ and
$$\rho=%\cfrac{\c}{\b-\cfrac{\c}{\b-\cfrac{\c}{\b-\cdots}}}
\frac{\c}{\b-}\;\frac{\c}{\b-}\;\frac{\c}{\b-}\;\frac{\c}{\b-}\;\cdots.$$
It follows that $\rho=\frac{\beta-\sqrt{\beta^2-4\gamma}}{2}=\frac{b-\sqrt{b^2-4ac}}{2a}=\la_1$.
This completes the proof of the ``only if" part.
\end{proof}

%remark

\begin{exm}
Let $b,c>0$ and the ration function
$$\frac{1}{1-bx+cx^2}=\sum_{n\ge 0}u_nx^n.$$
Then $u_0=1,u_1=b$ and $u_{n+1}=bu_n-cu_{n-1}$.
Thus all $u_n$ are positive
if and only if $b^2\ge 4c$,
a folklore result.

Similarly, let $b,c,d>0$ and
$$\frac{1-dx}{1-bx+cx^2}=\sum_{n\ge 0}u_nx^n.$$
Then $u_0=1,u_1=b-d$ and $u_{n+1}=bu_n-cu_{n-1}$.
Thus all $u_n$ are positive
if and only if $b^2\ge 4c$ and %$b+\sqrt{b^2-4c}\ge 2d$.
%$b-d\ge \la_1=(b-\sqrt{b^2-4c})/2$,
$d\le (b+\sqrt{b^2-4c})/2$.
%This result has occurred in \cite[Proposition 2.6]{CLW15ra}.
%and the proof there is not so natural.
\end{exm}

%\section{Applications of Theorem \ref{sc-pp}}

\begin{exm}\label{dtc-exm}
The question of determining whether Taylor coefficients of a given rational function are all positive,
%goes back to Szeg\"o (1933) and has since been investigated by many authors.
has been investigated by many authors \cite{Ask74,AG72,Kau07,Pil19,SS14,Str08,SZ15}.
In order to show the positivity of the rational functions,
it is necessary even suffices to prove that diagonal Taylor coefficients are positive.
The diagonal coefficients of some important rational functions
are arithmetically interesting sequences and satisfy three-term recurrence relations.
Straub and Zudilin \cite{SZ15} showed that these diagonal coefficients are positive
by expressed them in terms of known hypergeometric summations.
Here we show their positivity from the viewpoint of three-term recurrence sequences.
%in a unified approach.

(1)\quad
Consider the rational function
\begin{equation}\label{exm-a}
\frac{1}{1-(x+y)+axy}=\sum_{n,m\ge 0}u_{n,m}x^ny^m.
\end{equation}
The diagonal terms $u_n:=u_{n,n}$ of the Taylor expansion
satisfy the recurrence relation
$$(n+1)u_{n+1}=(2-a)(2n+1)u_n-a^2nu_{n-1}$$
with $u_0=1$ and $u_1=2-a$.
The characteristic function $Q(\la)=\la^2-2(2-a)\la+a^2$ and the discriminant $\Delta=16(1-a)$.
If $(u_n)$ is positive, then $\Delta\ge 0$ by Theorem \ref{nc-pp}, i.e., $a\le 1$.
Conversely, if $a\le 1$, then $\la_1=2-a-2\sqrt{1-a}$.
Clearly, $u_1=2-a\ge\la_1=\la_1 u_0$ and
$Q_n(\la_1)%=(n+1)\la_1^2-(2-a)(2n+1)\la_1+a^2n=Q(\la_1)n+\la_1[\la_1-(2-a)]
=\la_1[\la_1-(2-a)]\le 0$.
%The positivity of $(u_n)_{n\ge 0}$ follows from Theorem \ref{sc-pp} by taking $\la_0=\la_1$.
It follows that $(u_n)_{n\ge 0}$ is positive from Theorems \ref{sc-pp} by taking $\la_0=\la_1$. %Corollary \ref{1-rr}.
Thus we conclude that $(u_n)_{n\ge 0}$ is positive if and only if $a\le 1$.
It is also known that
$$u_n=\sum_{k=0}^{n}\frac{(2n-k)!}{k!(n-k)!^2}(-a)^k.$$
The positivity is not apparent when $0<a\le 1$.
%For $a>1$, the quadratic polynomial $1-2(2-a)z+a^2z^2$ has non-real roots,
%and so that $u_n$ is (eventually) sign-indefinite.
%Therefore, the series (9) is positive if and only if its diagonal terms are positive.
%It follows that all $u_n>0 \iff a\le 1$.
%Thus the rational function is positive iff its diagonal terms are positive.

%On the other hand,
Straub \cite[Proposition 4]{Str08} showed that $u_{n,m}$ in \eqref{exm-a} are all positive if and only if $a\le 1$.
In other words, the rational function \eqref{exm-a} is positive if and only if its diagonal terms are positive.

(2)\quad
Consider the Szeg\"o rational function
$$S(x,y,z)=\dfrac{1}{1-(x+y+z)+\frac{3}{4}(xy+yz+zx)}.$$
Denote the diagonal terms
$s_n=[(xyz)^n]S(2x,2y,2z)$.
It is known that
$$s_n=\sum_{k=0}^{n}(-27)^{n-k}2^{2k-n}\frac{(3k)!}{k!^3}\binom{k}{n-k},$$
the positivity is not apparent here.
On the other hand, the diagonal terms satisfy the three-term recurrence relation
$$2(n+1)^2s_{n+1}=3(27n^2+27n+8)s_n-81(3n-1)(3n+1)s_{n-1},$$
with $s_0=1, s_1=12$ and $s_2=198$.
The characteristic equation $2\la^2-81\la+729=0$ has two roots
$\la_1=27/2$ and $\la_2=27$.
Also, $s_2>\la_1 s_1$ and
$$Q_n(\la_1)=2(2n+1)\la_1^2-3(27n+8)\la_1-81=-\frac{729}{2}n-\frac{81}{2} %=-\frac{729}{2}(n-1)-405\le Q_1(\la_1)
<0$$
for $n\ge 1$.
The positivity of $(s_n)_{n\ge 1}$ follows from Theorem \ref{sc-pp} by taking $\la_0=\la_1$.
Thus the total sequence $(s_n)_{n\ge 0}$ is positive.

(3)\quad
Consider the Lewy-Askey rational function
$$h(x,y,z,w)=\dfrac{1}{1-(x+y+z+w)+\frac{2}{3}(xy+xz+xw+yz+yw+zw)}.$$
Let $t_n=9^n[(xyzw)^n]h(x,y,z,w)$ and $t_n=\binom{2n}{n}h_n$.
Then $h_0=1,h_1=24$ and
$$3(n+1)^2h_{n+1}=4(28n^2+28n+9)h_n-64(4n-1)(4n+1)h_{n-1}.$$
The characteristic equation $3\la^2-112\la+1024=0$ has two roots
$\la_1=16$ and $\la_2=64/3$.
Also,
$$Q_n(\la_1)=3(2n+1)\la_1^2-4(28n+9)\la_1-64=-256(n-1)-128<0$$
for $n\ge 1$, and $h_1>\la_1 h_0$.
The positivity of $(h_n)_{n\ge 0}$ follows from Theorem \ref{sc-pp} by taking $\la_0=\la_1$.

(4)\quad
Consider the Kauers-Zeilberger rational function
$$D(x,y,z,w)=\dfrac{1}{1-(x+y+z+w)+2(xyz+xyw+xzw+yzw)+4xyzw}.$$
Let $d_n=[(xyzw)^n]D(x,y,z,w)$.
Then
$$(n+1)^3d_{n+1}=4(2n+1)(3n^2+3n+1)d_n-16n^3d_{n-1}$$
with $d_0=1$ and $d_1=4$.
The characteristic equation $\la^2-24\la+16=0$ has two roots
$\la_1=12-8\sqrt{2}<1<\la_2=12+8\sqrt{2}$.
The positivity of $(d_n)_{n\ge 0}$ follows from Theorem \ref{sc-pp} by taking $\la_0=1$
since $b(n)\ge a(n)+c(n)$.
\end{exm}

\begin{rem}\label{aln}
%In his remarkable proof of the irrationality of $\zeta(3)=\sum_{n\ge 1}1/n^3$,
The Ap\'ery numbers
$$A_n=\sum_{k=0}^n \binom{n}{k}^2\binom{n+k}{k}^2$$
play an important role in Ap\'ery's proof of the irrationality of $\zeta(3)=\sum_{n\ge 1}1/n^3$.
The Ap\'ery numbers are diagonal Taylor coefficients of the rational function
$$%R(x,y,z,w)=
\frac{1}{1-(xyzw+xyw+xy+xz+zw+y+z)}$$
and satisfy the three-term recurrence relation
\begin{equation}\label{a-33}
(n+1)^3A_{n+1}=(2n+1)(17n^2+17n+5)A_n-n^3A_{n-1}.
\end{equation}
We refer the reader to \cite[A005259]{Slo} and references therein for the Ap\'ery numbers.
The Ap\'ery numbers are closely related to modula forms or supercongruences %and the $p$-adic Gamma function (ZWS2010)
and have been generalized to various Ap\'ery-like numbers,
which satisfy three-term recurrence relations similar to \eqref{a-33}
(see \cite{Coo12,MS16} for instance).
%For example,
The diagonal terms $s_n,h_n$ and $d_n$ in Example \ref{dtc-exm} are all Ap\'ery-like numbers.
%It is known that the sequence $(A_n)_{n\ge 0}$ is log-convex. (?)
%We next consider the positivity and log-convexity of Ap\'ery-like numbers. %defined in Malik and Straub \cite{MS16}.
Not all Ap\'ery-like numbers are positive.
For example, consider the Ap\'ery-like numbers $(u_n)_{n\ge 0}$ defined by
$$u_n=\sum_{k=0}^{\lrf{n/3}}
(-1)^k3^{n-3k}\binom{n}{3k}\frac{(3k)!}{k!^3},$$
which are diagonal Taylor coefficients of the rational function
$$\frac{1}{1+x^3+y^3+z^3-3xyz}$$ %=\sum_{n,m\ge 0}u_{n,m}x^ny^m$$
and satisfy the recurrence relation
$$(n+1)^2u_{n+1}=(9n^2+9n+3)u_n-27n^2u_{n-1}.$$%\qquad\text{(A006077)}.$$
See \cite[A006077]{Slo} and references therein.
Note that the discriminant of the characteristic equation $\la^2-9\la+27=0$ is negative.
%$\Delta=9^2-4\cdot 27<0$.
Hence the sequence $(u_n)$ is oscillatory by Theorem \ref{nc-pp} (i). %(eventually) sign-indefinite.
\end{rem}

A sequence $(u_n)_{n\ge 0}$ of positive numbers is said to be {\it log-convex}
if $u_{n-1}u_{n+1}\ge u_n^2$ for all $n\ge 1$.
The log-convexity of combinatorial sequences have been extensively investigated
(see \cite{LW07lcx} for instance).
Here we present a new criterion,
which can be simultaneously used for the positivity and log-convexity of three-term recurrence sequences.

Denote
$$B(n)=\left|
        \begin{array}{cc}
          b(n+1) & b(n)\\
          a(n+1) & a(n)\\
        \end{array}
      \right|=Bn^{2\delta-2}+\cdots,\quad
C(n)=\left|
        \begin{array}{cc}
          c(n+1) & c(n)\\
          a(n+1) & a(n)\\
        \end{array}
      \right|=Cn^{2\delta-2}+\cdots
$$
where
$$B=\left|\begin{array}{cc}
b & b'\\
a & a'\\
\end{array}
\right|=ba'-b'a,\quad
C=\left|\begin{array}{cc}
c & c'\\
a & a'\\
\end{array}
\right|=ca'-c'a.$$

  %We refer the reader to \cite[Theorem 3.10]{LW07lcx} for the special case $\delta=1$ in Theorem \ref{lcx}.
%The special case $\delta=1$ of Theorem \ref{lcx} has occurred in \cite[Theorem 3.10]{LW07lcx}.
  %The result of the special case $\delta=1$ in Theorem \ref{lcx} has occurred in \cite[Theorem 3.10]{LW07lcx}.

\begin{prop}[Log-convexity]
\label{lcx}
Let $(u_n)_{n\ge0}$ be a sequence satisfying the recurrence relation (\ref{u-rr}).
%Suppose that $B,C>0$ and $CB(n)\ge BC(n)\ge 0$ for $n\ge 1$.
%Let $\la_0=C/B$.
%Assume that $u_1\ge\la_0 u_0$ and $Q_n(\la_0)\le 0$ for $n\ge 1$.
%If $(u_0,u_1,u_2)$ is positive and log-convex,
%then so is the total sequence $(u_n)_{n\ge 0}$.
  %Assume that there exists some $m$ such that $u_m,u_{m+1}, u_{m+2}$ are positive and $u_{m+2}/u_{m+1}\ge u_{m+1}/u_m\ge\la_0$
  %and $Q_n(\la_0)\le 0$ for $n\ge m$.
  %Then the sequence $(u_n)_{n\ge m}$ is positive and log-convex.
Suppose that $B,C>0$ and let $\la_0=C/B$.
\begin{enumerate}[\rm (i)]
  \item Assume that $u_1\ge\la_0 u_0>0$ and $Q_n(\la_0)\le 0$ for $n\ge 1$.
  Then the sequence $(u_n)_{n\ge 0}$ is positive.
  \item Assume that the sequence $(u_n)_{n\ge 0}$ is positive and $CB(n)\ge BC(n)\ge 0$ for $n\ge 1$.
  If $u_{2}/u_{1}\ge u_{1}/u_0\ge\la_0$,
  then the sequence $(u_n)_{n\ge 0}$ is log-convex.
\end{enumerate}
\end{prop}
\begin{proof}%[\bf Proof of Theorem \ref{lcx}]
(i)\quad
The positivity of $(u_n)_{n\ge 0}$ is obvious by Theorem \ref{sc-pp}.

(ii)\quad
Let $x_n={u_{n+1}}/{u_n}$ for $n\ge 0$.
Then $(u_n)_{n\ge 0}$ is log-convex if and only if $(x_n)_{n\ge 0}$ is nondecreasing.
We next show that $x_{n+1}\ge x_{n}\ge\la_0$ for $n\ge 0$.
We proceed by induction on $n$.
Clearly, $x_1\ge x_0\ge\la_0$. %by the condition. %$u_0u_2\ge u_1^2$ and $u_1\ge\la_0u_0$.
Assume now that $x_n\ge x_{n-1}\ge\la_0$.
We need to show that $x_{n+1}\ge x_n\ge\la_0$.

By the recurrence relation \eqref{u-rr}, we have
\begin{equation}\label{xn-rr0}
x_n=\frac{b(n)}{a(n)}-\frac{c(n)}{a(n)}\frac{1}{x_{n-1}}.
\end{equation}
Thus
\begin{eqnarray}\label{xn1}
x_{n+1}-x_n
&=& \left[\left(\frac{b(n+1)}{a(n+1)}-\frac{b(n)}{a(n)}\right)
- \left( \frac{c(n+1)}{a(n+1)}-\frac{c(n)}{a(n)}\right) \frac{1}{x_n}\right]
+ \frac{c(n)}{a(n)} \left( \frac{1}{x_{n-1}}-\frac{1}{x_n} \right)\nonumber\\
&=& \frac{B(n)x_n-C(n)}{a(n+1)a(n)x_n}+\frac{c(n)}{a(n)} \left( \frac{1}{x_{n-1}}-\frac{1}{x_n} \right).
\end{eqnarray}
By the assumption $x_n\ge\la_0$ and the condition $CB(n)\ge BC(n)$, %and $\la_0=C/B$,
we obtain $B(n)x_n\ge B(n)\la_0\ge C(n)$.
It follows from \eqref{xn1} that $x_{n+1}\ge x_n$, as required.
Thus the sequence $(x_n)_{n\ge 0}$ is nondecreasing,
and the sequence $(u_n)_{n\ge 0}$ is therefore log-convex.
\end{proof}

By means of Proposition \ref{lcx},
we may prove that the log-convexity of the diagonal terms $s_n,h_n,d_n$ in Example \ref{dtc-exm},
as well as the Ap\'ery numbers $A_n$.
We omit the proofs for brevity.
Instead
we give a somewhat more complex example to illustrate Proposition \ref{lcx}.
%we illustrate Theorem \ref{lcx} in a.
%We can also investigate the log-convexity of Ap\'ery-like numbers.
%We refer the reader to \cite{AZ06,Coo12,MS16} for Ap\'ery-like numbers.

\begin{exm}
Consider the Ap\'ery-like numbers $(u_n)_{n\ge 0}$ defined by
\begin{equation}\label{s18}
(n+1)^3u_{n+1}=(2n+1)(14n^2+14n+6)u_n-n(192n^2-12)u_{n-1}
\end{equation}
with $u_0=1$ and $u_1=6$. %$u_2=54,u_3=564,u_4=6390,u_5=76356,u_6=948276$.
%Now $\la^2-28\la+192=0$ has two roots
%$\la_1=12$ and $\la_2=16$, and
%$$Q_n(\la_1) %=-84n^2+120n+72=-12(7n^2-10n-6)
%=-72n^2+108n+72=-36(n-2)(2n+1)\le 0$$
%for $n\ge 2$.
%Also, $u_6>12u_5$.
Such Ap\'ery-like numbers are introduced by Cooper in \cite{Coo12}.
It is known that
$$u_n= %s_{18}(n)=
\sum_{k=0}^{\lrf{n/3}}
(-1)^k\binom{n}{k}\binom{2k}{k}\binom{2(n-k)}{n-k}\left[\binom{2n-3k-1}{n}+\binom{2n-3k}{n}\right].$$
%The positivity is not apparent here
%but can be obtained by means of Theorem \ref{sc-pp}.
We next apply Proposition \ref{lcx} to obtain the positivity and log-convexity simultaneously.

We have
$$B(n)=42n^4+200n^3+330n^2+220n+54$$
and
$$C(n)=576n^4+2328n^3+2952n^2+1200n+180.$$
Hence $B=42, C=576$ and
$$CB(n)-BC(n)=17424n^3+66096n^2+76320n+23544$$
for $n\ge 1$.
On the other hand,
$\la_0=C/B=96/7$ and
$$Q_n(\la_0)=\frac{12}{49}(-16n^3-48n^2+799n+432)<0$$
for $n\ge 7$.
Also, $u_{12}/u_{11}\ge u_{11}/u_{10}>\la_0$.
The sequence $(u_n)_{n\ge 10}$ is therefore positive and log-convex by Proposition \ref{lcx}.
It is not difficult to check that $(u_n)_{0\le n\le 11}$ is also positive and log-convex.
Thus the total sequence $(u_n)_{n\ge 0}$ is positive and log-convex.
\end{exm}

%\begin{rem}
We also refer the interested reader to
\cite{XY13} for the log-convexity of three-term recursive sequences and
\cite{HZ19} for the asymptotic log-convexity of $P$-recursive sequences.
%\end{rem}

\section{Concluding remarks and further work}

We have seen that
if $b^2<4ac$,
then the difference equation \eqref{u-rr} is oscillatory;
and if $b^2>4ac$,
then the difference equation \eqref{u-rr} is nonoscillatory. %eventually sign-definite.
%For the case $b^2=4ac$,
%the difference equation \eqref{u-rr} may be either eventually sign-definite or oscillatory.
%(The general recursive formula for ? is too complicated to be included here. Page 378)
In the case $b^2=4ac$,
the asymptotic behavior of solutions of the second-order difference equations can be very complicated.
%too complicated to be included here.
%We refer the interested reader to %\cite{BT33,Ela05,WZ85,WL92a,WL92b}
%\cite{WL92a,WL92b} for further information.
The interested reader is referred to Wong and Li \cite{WL92a,WL92b}.
Here we illustrate that the difference equation \eqref{u-rr} may be either oscillatory or nonoscillatory.
%For example, every solution of the difference equation $u_{n+1}=2u_n-u_{n-1}$ is eventually sign-definite
%since $u_n=n(u_1-u_0)+u_0$.
%On the other hand, the difference equation
%$nu_{n+1}=(2n-1)u_n-nu_{n-1}$ is oscillatory \cite[Exercise 7.1.3]{Ela05}.

\begin{exm}\label{D=0}
Consider the difference equation
\begin{equation}\label{Lnx}
(n+1)L_{n+1}(x)=(2n+1-x)L_n(x)-nL_{n-1}(x).
\end{equation}
Clearly, the corresponding discriminant $b^2-4ac=0$.

When $x=0$, we have $(n+1)L_{n+1}(0)=(2n+1)L_n(0)-nL_{n-1}(0)$.
Every solution of this difference equation is nonoscillatory.
Actually, solve the difference equation to obtain
\begin{equation}\label{Ln0}
L_n(0)=\left(1+\frac{1}{2}+\cdots+\frac{1}{n}\right)(L_1(0)-L_0(0))+L_0(0).
\end{equation}
Recall that $1+\frac{1}{2}+\cdots+\frac{1}{n}\sim \ln n+\gamma$,
where $\gamma$ is the Euler constant.
Hence if $L_1(0)<L_0(0)$, then $L_n(0)$ is eventually negative;
if $L_1(0)=L_0(0)$, then $L_n(0)$ are identically equal to $L_0(0)$; %for all $n$.
and if $L_1(0)>L_0(0)$, then $L_n(0)$ is eventually positive.
In case of positive,
it immediately follows from \eqref{Ln0} that the sequence $(L_n(0))$ is concave, and therefore log-concave.

%Recall that $$1+\frac{1}{2}+\cdots+\frac{1}{n}\sim\ln n+\gamma,$$
%where $\gamma$ is the Euler constant.
%If $L_1(0)>L_0(0)$,
%then the sequence $L_n(0)$ is (asymptotically) concave, and is therefore log-concave.

When $x=1$, we have
\begin{equation}\label{Ln1-rr}
(n+1)L_{n+1}(1)=2nL_n(1)-nL_{n-1}(1).
\end{equation}
We next show that every solution of the difference equation \eqref{Ln1-rr} is oscillatory.

Suppose the contrary and let $L_n$ be an eventually positive solution of \eqref{Ln1-rr}.
We may assume, without loss of generality, that
$L_n>0$ for all $n\ge 0$.
Let $x_n=L_{n+1}/L_n$ for $n\ge 0$.
Then
\begin{equation}\label{xn-rr}
  x_n=\frac{2n}{n+1}-\frac{n}{n+1}\frac{1}{x_{n-1}}=\frac{n}{n+1}\left(2-\frac{1}{x_{n-1}}\right).
\end{equation}
Note that $x_1=1-\frac{1}{2x_0}<1$.
Assume that $x_{n-1}<1$.
Then $x_n<\frac{2n}{n+1}-\frac{n}{n+1}=\frac{n}{n+1}<1$ by \eqref{xn-rr}.
Thus $x_n<1$ for all $n\ge 1$.
On the other hand, since $a+1/a\ge 2$ for $a>0$, we have
$$x_n=\frac{n}{n+1}\left(2-\frac{1}{x_{n-1}}\right)\le\frac{n}{n+1}x_{n-1}<x_{n-1}.$$
The sequence $(x_n)_{n\ge 1}$ is therefore decreasing.
Thus the sequence $(x_n)$ is convergence.
Let $x_n\rightarrow x$.
Then $x=2-1/x$ by \eqref{xn-rr}, and so $x=1$.
On the other hand, $x<x_1<1$ since $(x_n)$ is decreasing,
which leads to a contradiction.

%It is well known that %signless
The classic Laguerre polynomials
$$L^{(0)}_n(x)=\sum_{k=0}^n(-1)^k\binom{n}{k}\frac{x^k}{k!}$$
satisfy the recurrence relation \eqref{Lnx}
with $L^{(0)}_0(x)=1$ and $L^{(0)}_1(x)=1-x$.
%It is well known that the Laguerre polynomials satisfy the Turan inequality:
%$$L^2_n(x)-L_{n-1}(x)L_{n+1}(x)\ge 0$$
%for $x\in\mathbb{R}$ and $n\ge 1$.
It is well known that
%$$L^{(0)}_n(1)=\pi^{-1/2}e^{1/2}n^{-1/4}\cos\left(2n^{1/2}-1/4\right)+O\left(n^{-3/4}\right).$$
$$L^{(0)}_n(x)=\pi^{-1/2}e^{x/2}(nx)^{-1/4}\cos\left(2(nx)^{1/2}-1/4\right)+O\left(n^{-3/4}\right)$$
(see \cite[Example 8.38]{Ela05} for instance).
This can also explain why $L^{(0)}_n(1)$ is oscillatory.
\end{exm}

For the difference equation \eqref{u-rr} with $b^2=4ac$,
we feel that either all solutions are oscillatory or
all solutions are nonoscillatory.
However, we can not prove it.

%We end this paper by proposing the following problem.
We have seen from Theorem \ref{nc-pp} (ii) that
if the difference equation \eqref{u-rr} has a positive solution,
then it has a positive and minimal solution $u^*_n$.
%By Poincar\'e Theorem,
%Our last question is under what conditions,
We conjecture that the solution $u^*_n$ is log-convex
and the ratio $u^*_{n+1}/u^*_n$ converges to the smaller characteristic root $\la_1$.

\section*{Acknowledgement}

This work was partially supported by the National Natural Science Foundation of China (Nos. 11771065, 12171068).

\end{document}